\documentclass[12pt]{amsart}
\usepackage{stmaryrd}
\usepackage{mathrsfs}
\usepackage{amsmath,amssymb}
\usepackage{amsfonts}
\usepackage{amsthm}
\usepackage{latexsym}
\usepackage{graphicx,color}
\def\supp{\mbox{supp}}

\def\p{\partial}

\def\R{\mathbb{R}}

\def\vv<#1>{\langle#1\rangle}

\def\1{\mathbf{1}}

\def\XXint#1#2{\setbox0=\hbox{$#1{#2}{\int}$}{#2}\kern-.5\wd0 }

\def\XXint#1#2#3{{\setbox0=\hbox{$#1{#2#3}{\int}$}
     \vcenter{\hbox{$#2#3$}}\kern-.5\wd0}}

\def\vv<#1>{{\left\langle#1\right\rangle}}

\def\Z{\mathbb{Z}}

\newtheorem{thm}{Theorem}[section]

\newtheorem{cor}{Corollary}[section]
\theoremstyle{definition}

\theoremstyle{remark}

\numberwithin{equation}{section}

\begin{document}
\title{Sharp Gagliardo-Nirenberg inequality and logarithmic Sobolev inequality on integer lattices}

\author{Yongjie Shi$^1$}
\address{Department of Mathematics, Shantou University, Shantou, Guangdong, 515063, China}
\email{yjshi@stu.edu.cn}
\author{Chengjie Yu$^2$}
\address{Department of Mathematics, Shantou University, Shantou, Guangdong, 515063, China}
\email{cjyu@stu.edu.cn}
\thanks{$^1$Research partially supported by GDNSF with contract no. 2025A1515011144 and NNSF of China with contract no. 11701355. }
\thanks{$^2$ Research partially supported by GDNSF with contract no. 2025A1515011144.}
\renewcommand{\subjclassname}{%
  \textup{2020} Mathematics Subject Classification}
\subjclass[2020]{Primary 39A12; Secondary 05C50}
\date{}
\keywords{Garliardo-Nirenberg inequality, logarithmic Sobolev inequality, integer lattice}
\begin{abstract}
In this paper, we obtain a sharp Garliardo-Nirenberg inequality on integer lattices and characterize its rigidity. Moreover, as a consequence of the sharp Garliardo-Nirenberg inequality, we obtain sharp logarithmic Sobolev inequalities on integer lattices. 
\end{abstract}
\maketitle
\markboth{Shi \& Yu}{Sharp Gagliardo-Nirenberg inequality and logarithmic Sobolev inequality}
\section{Introduction}
In discrete geometric analysis, geometric inequalities such as the Sobolev inequality and the isoperimetric inequality play important roles. For examples, in the recent works \cite{HL,HLM}, the authors considered the existence of extremal functions for Sobolev inequalities on $\Z^n$. In \cite{SYZ}, the author considered  Sobolev inequalities on general weighted graphs. In \cite{Jo,LLS,Mu}, the authors  considered logarithmic Sobolev inequalities on graphs with some curvature assumptions. 
   
In this paper, we obtain the following discrete version of the Garliardo-Nirenberg inequality (see \cite{GT,Ga,Ni,Sa}) on $\Z^n$, and characterize its rigidity.
\begin{thm}\label{thm-G-N}
For any $f\in C_0(\Z^n)$ with $n\geq 2$,
\begin{equation}\label{eq-G-N}
\|f\|_{\frac{n}{n-1}}\leq\frac12\prod_{i=1}^n\|\p_if\|_1^\frac1n.
\end{equation}
Moreover, the equality holds if and only if $f=\lambda\cdot\chi_{A}$ where $A$ is a cuboid in $\Z^n$.
\end{thm} 
By applying the AM-GM inequality to the RHS of the sharp Garliardo-Nirenberg inequality \eqref{eq-G-N}, one obtains the following sharp Sobolev inequality.
\begin{cor}\label{cor-Sobolev}
For any $f\in C_0(\Z^n)$ with $n\geq 2$,
\begin{equation}\label{eq-Sobolev}
\|f\|_{\frac{n}{n-1}}\leq \frac{1}{2n}\|df\|_1.
\end{equation}
Moreover, the equality holds if and only if $f=\lambda\cdot\chi_Q$ where $Q$ is a cube in $\Z^n$. 
\end{cor}

In \eqref{eq-Sobolev}, letting $f=\chi_A$ for any finite set $A\subset \Z^n$, one obtains the sharp isoperimentric inequality in $\Z^n$:
\begin{equation}\label{eq-isoperimetric}
|A|^{n-1}\leq \frac{1}{2^nn^n}|\p A|^n
\end{equation}
with that the equality holds if and only if $A$ is a cube in $\Z^n$. The sharp isoperimetric inequality \eqref{eq-isoperimetric} can be also found in \cite{LP} which was proved by using the discrete part of the Loomis-Whitney inequality \cite{LW}. Moreover, it is well known that \eqref{eq-Sobolev} and \eqref{eq-isoperimetric} are equivalent (see \cite{Sa}). 

On the other hand, as another consequence the sharp Garliardo-Nirenberg inequality \eqref{eq-G-N}, we have the following sharp logarithmic Sobolev inequalities on $\Z^n$.
\begin{cor}\label{cor-log-Sobolev-1}
Let $n\geq 2$, $p>0$ and  $f\in C_0(\Z^n)$ be nonnegative  with $\|f\|_p=1$. Then
\begin{equation}\label{eq-log-Sobolev-1}
\left(\frac1n+\frac1p-1\right)\int_{\Z^n}f^p\log f^pdz_1dz_2\cdots dz_n\leq-\log 2+\frac1{n} \sum_{i=1}^n\log\|\p_if\|_1
\end{equation}
with that the equality holds if and only if $f=|A|^{-\frac1p}\cdot\chi_A$ for some cuboid $A$ in $\Z^n$.
\end{cor}
\begin{cor}\label{cor-log-Sobolev-2}
Let $n\geq 2$, $p>0$ and  $f\in C_0(\Z^n)$ be nonnegative  with $\|f\|_p=1$. Then
\begin{equation}\label{eq-log-Sobolev-2}
\left(\frac1n+\frac1p-1\right)\int_{\Z^n}f^p\log f^pdz_1dz_2\cdots dz_n\leq\log\left(\frac1{2n}\|df\|_1\right)
\end{equation}
with that the equality holds if and only if $f=|Q|^{-\frac1p}\chi_Q$ for some cube $Q$ in $\Z^n$.
\end{cor}
In the logarithmic Sobolev inequalities \eqref{eq-log-Sobolev-1} and \eqref{eq-log-Sobolev-2}, if we let $p=\frac{n-1}{n}$, then we recover the Garliardo-Nirenberg inequality \eqref{eq-G-N} and the sharp Sobolev inequality \eqref{eq-Sobolev} respectively. Moreover, if we let $p=1$ in \eqref{eq-log-Sobolev-2}, then \eqref{eq-log-Sobolev-2} is exactly the discrete version of the $L^1$ logarithmic Sobolev inequality on $\R^n$ (see \cite{Ge}).

Motivated by the work \cite{CGS}, we prove Theorem \ref{thm-G-N} by first proving the following discrete Brascamp-Lieb inequality and characterizing its rigidity.  
\begin{thm}\label{thm-BL}
For any nonnegative $f\in C_0(\Z^n)$ with $n\geq 2$, let 
$$f_i(z_1,z_2,\cdots,\widehat{z_i},\cdots,z_{n})=\max_{z_i\in \Z} f(z_1,z_2,\cdots,z_i,\cdots,z_n).$$
Here $\widehat{z_i}$ means that $z_i$ is not appeared in the expression. Then, 
\begin{equation}\label{eq-BL}
\|f\|_{\frac{n}{n-1}}\leq \left(\prod_{i=1}^n\|f_i\|_{1}\right)^{\frac{1}{n}}.
\end{equation}
Moreover, the equality of \eqref{eq-BL} holds if and only if $f=\lambda\cdot \chi_{A}$ with $A=\prod_{i=1}^n\pi_i(A)$ and $\lambda> 0$. 
\end{thm}
As a direct consequence of the discrete Brascamp-Lieb inequality \eqref{eq-BL}, we have the following logarithmic-type Brascamp-Lieb inequality. 
\begin{cor}\label{cor-log-BL}
Let $n\geq 2$, $p>0$, and $f\in C_0(\Z^n)$ be nonnegative  with $\|f\|_p=1$. 
Then, 
\begin{equation}\label{eq-log-BL}
\left(\frac1n+\frac1p-1\right)\int_{\Z^n}f^p\log f^pdz_1dz_2\cdots dz_n\leq \sum_{i=1}^n\log\|f_i\|_{1}.
\end{equation}
Moreover, the equality of \eqref{eq-log-BL} holds if and only if $f=|A|^{-\frac1p} \cdot\chi_{A}$ with $A=\prod_{i=1}^n\pi_i(A)$. 
\end{cor} 
The same as before, when letting $p=\frac{n}{n-1}$ in \eqref{eq-log-BL}, we recover \eqref{eq-BL}. Moreover, when letting $f=\chi_A$ with $A\subset \Z^n$ in \eqref{eq-BL}, we obtain the discrete part of the Loomis-Whitney inequality:
\begin{equation}\label{eq-LM}
|A|^{n-1}\leq \prod_{i=1}^n\pi_i(A)
\end{equation} 
for any finite subset $A$ in $\Z^n$ with that the equality holds if and only if $A=\prod_{i=1}^n\pi_i(A)$. Finally, we would like to mention that the discrete Brascamp-Lieb inequality has been extensively studied in \cite{CD}.

The rest of the paper is organized as follows. In Section 2, we give some preliminaries for analysis on $\Z^n$. In Section 3, we prove the main results of the paper. 
\section{Preliminaries}
In this section, we fix some notations for analysis on $\Z^n$. 

We consider $\Z^n$ as a graph with the set of vertices $\Z^n$ and with the set of edges
$$E(\Z^n)=\left\{\{z,z+e_i\}\ |\ z\in \Z^n,\ i=1,2,\cdots,n\right\}.$$
Here $e_i\in \Z^n$ is a vector with the $i^{\rm th}$ component being $1$ and the other components vanished for $i=1,2,\cdots,n$.

We denote $\pi_i:\Z^n\to \Z$ the natural projection sending $(z_1,z_2,\cdots,z_n)$ to $z_i$. The counting measure on $\Z$ is denoted as $dz$ and its product on $\Z^n$ is denoted as $dz_1dz_2\cdots dz_n$ which is also the counting measure on $\Z^n$. We denote the space of functions on $\Z^n$ as $C(\Z^n)$ and the space of functions on $\Z^n$ with finite support as $C_0(\Z^n)$. 

For any $p>0$ and $f\in C_0(\Z^n)$, define
$$\|f\|_p:=\left(\int_{\Z^n}|f|^pdz_1dz_2\cdots dz_n\right)^\frac{1}{p}.$$
For $i=1,2,\cdots,n$ and $f\in C(\Z^n)$, define
$$\p_if(z):=f(z+e_i)-f(z),\ \forall\ z\in \Z^n.$$

Moreover, on a graph $G=(V,E)$, a 1-form $\alpha$ is a bivariate asymmetric function on $V$ such that 
$$\alpha(x,y)=0$$
if $\{x,y\}\not\in E$. The space of 1-forms on $G$ is denoted as $A^1(G)$. The differential of a function $f\in \R^V$ is defined as 
$$df(x,y):=\left\{\begin{array}{ll}f(y)-f(x)&\{x,y\}\in E\\
0&\mbox{otherwise.}
\end{array}\right.$$
The space of 1-forms with finite support on $G$ is denoted as $A^1_0(G)$. For $p>0$ and $\alpha\in A_0^1(G)$, define 
$$\|\alpha\|_p:=\left(\sum_{\{x,y\}\in E}|\alpha(x,y)|^p\right)^\frac1p.$$ 
It is clear by definition that 
$$\|df\|_p=\left(\sum_{i=1}^n\|\p_if\|_p^p\right)^\frac{1}{p}$$
for all $p>0$ and $f\in C_0(\Z^n)$. For any subset $A$ of $V$, we define
$$\p A=\{\{x,y\}\in E\ |\ x\in A,\ y\in A^c\}.$$

Finally, for $a,b\in \Z$ with $a\leq b$, we denote 
$$[a,b]=\{z\in \Z\ |\ a\leq z\leq b\}.$$
A set  
$$A=\prod_{i=1}^n[a_i,b_i]\subset \Z^n$$
is called a cuboid in $\Z^n$. If moreover, $|b_i-a_i|$ is independent of $i$, we call $A$ a cube in $\Z^n$.
\section{Proofs of main results}
In this section, we prove the main results of the paper. We first prove Theorem \ref{thm-BL}, a discrete version of the Brascamp-Lieb inequality, by induction.

\begin{proof}[Proof of Theorem \ref{thm-BL}]
We proceed by induction. When $n=2$, by that 
$$f(z_1,z_2)\leq f_1(z_2)\mbox{ and }f(z_1,z_2)\leq f_2(z_1),$$
we have 
\begin{equation}\label{eq-square-f}
f^2(z_1,z_2)\leq f_1(z_2)f_2(z_1).
\end{equation}
Thus
$$\int_{\Z^2}f^2(z_1,z_2)dz_1dz_2\leq \int_{\Z^2} f_1(z_2)f_2(z_1)dz_1dz_2=\int_{\Z} f_1(z_2)dz_2\cdot\int_{\Z} f_2(z_1)dz_1.$$
This gives us \eqref{eq-BL} for $n=2$. If the equality of $\eqref{eq-BL}$ holds for $n=2$. Then, the equality of \eqref{eq-square-f} holds for all $(z_1,z_2)\in \Z^2$. So, when $f(z_1,z_2)>0$, we have 
\begin{equation}\label{eq-f-f1-f2}
f(z_1,z_2)=f_1(z_2)=f_2(z_1).
\end{equation}
When $f(z_1,z_2)=0$, we have 
\begin{equation}
f_1(z_2)=0\mbox{ or } f_2(z_1)=0
\end{equation}
which means that 
\begin{equation}\label{eq-f1-f2}
f_1(\zeta,z_2)=0\ (\forall\ \zeta\in \Z)\mbox{ or }f_2(z_1,\zeta)=0\ (\forall\ \zeta\in \Z).
\end{equation}
If $f\equiv 0$, there is nothing to be proved. When $f\not\equiv 0$, let $(z_1,z_2), (w_1,w_2)\in \supp f$, by  \eqref{eq-f1-f2}, we know that $f(z_1,w_2)>0$. Moreover by \eqref{eq-f-f1-f2}, we have 
$$f(z_1,z_2)=f(z_1,w_2)=f(w_1,w_2).$$
So, $f=\lambda\cdot\chi_A$ with $A=\supp f$ and
$$|A|=|\pi_1(A)|\cdot|\pi_2(A)|.$$
Because $A\subset \pi_1(A)\times\pi_2(A)$, we have $A=\pi_1(A)\times \pi_2(A)$. This proves the conclusion of the theorem for $n=2$.

Suppose that the conclusion of the theorem is true for some $n\geq 2$. Then, for any 
nonnegative $f\in C_0(\Z^{n+1})$, 
\begin{equation}\label{eq-int-f-1}
\begin{split}
&\left(\int_{\Z^n} f^{\frac{n}{n-1}}(z_1,\cdots,z_n,z_{n+1})dz_1\cdots dz_n\right)^{n-1}\\
\leq& \prod_{i=1}^{n}\int_{\Z^{n-1}} f_i(z_1,\cdots,\widehat z_i,\cdots,z_n,z_{n+1})dz_1\cdots\widehat{dz_{i}}\cdots dz_n.
\end{split}
\end{equation}
Moreover, note that 
\begin{equation}\label{eq-int-f-2}
\begin{split}
&\int_{\Z^n} f(z_1,\cdots,z_n,z_{n+1})dz_1dz_2\cdots dz_n\\
\leq& \int_{\Z^n}f_{n+1}(z_1,z_2,\cdots,z_n)dz_1dz_2\cdots dz_n=\|f_{n+1}\|_{1},
\end{split}
\end{equation}
and by H\"older's inequality, 
\begin{equation}\label{eq-int-f-3}
\begin{split}
&\int_{\Z^n} f^{\frac{n+1}{n}}(z_1,\cdots,z_n,z_{n+1})dz_1\cdots dz_n\\
\leq&\left( \int_{\Z^n} f^{\frac{n}{n-1}}f(z_1,\cdots,z_n,z_{n+1})dz_1\cdots dz_n\right)^{\frac{n-1}{n}}\left(\int_{\Z^n} f(z_1,\cdots,z_n,z_{n+1})dz_1\cdots dz_n\right)^\frac{1}{n}.
\end{split}
\end{equation}
By combining \eqref{eq-int-f-1}, \eqref{eq-int-f-2} and \eqref{eq-int-f-3}, we have
\begin{equation}\label{eq-int-f-4}
\begin{split}
&\int_{\Z^n} f^{\frac{n+1}{n}}(z_1,\cdots,z_n,z_{n+1})dz_1\cdots dz_n\\
\leq& \prod_{i=1}^{n}\left(\int_{\Z^{n-1}} f_i(z_1,\cdots\widehat{z_i},\cdots,z_n,z_{n+1})dz_1\cdots\widehat{dz_i}\cdots dz_n\right)^{\frac1n}\|f_{n+1}\|_1^\frac{1}{n}.
\end{split}
\end{equation}
Thus, by using H\"older's inequality again, 
\begin{equation}
\begin{split}
&\int_{\Z^{n+1}} f^{\frac{n+1}{n}}(z_1,\cdots,z_n,z_{n+1})dz_1\cdots dz_ndz_{n+1}\\
\leq& \int_{\Z}\prod_{i=1}^{n}\left(\int_{\Z^{n-1}} f_i(z_1,\cdots,\widehat z_i,\cdots,z_n,z_{n+1})dz_1\cdots \widehat{dz_i}\cdots dz_n\right)^{\frac1n}dz_{n+1}\cdot\|f_{n+1}\|_1^\frac{1}{n}\\
\leq&\prod_{i=1}^{n+1}\|f_i\|_1^\frac1n.
\end{split}
\end{equation}
This gives us the inequality \eqref{eq-BL} for $n+1$. When the equality holds, if $f\equiv 0$, then there is nothing to be proved. If $f\not\equiv 0$, then $\|f_{n+1}\|_1> 0$. By that the equality of \eqref{eq-int-f-4} holds, we know that the equality of \eqref{eq-int-f-1} and \eqref{eq-int-f-2} holds. By the induction hypothesis and that the equality of \eqref{eq-int-f-1} holds, for any $z_{n+1}\in \Z$, 
\begin{equation}\label{eq-f-Q-1}
f(\cdot,z_{n+1})=\lambda_{z_{n+1}}\cdot\chi_{A_{z_{n+1}}}(\cdot)
\end{equation}
for some $A_{z_{n+1}}$ in $\Z^n$ such that $A_{z_{n+1}}=\prod_{i=1}^n\pi_i(A_{z_{n+1}})$ and $\lambda_{z_{n+1}}>0$.  By that the equality of \eqref{eq-int-f-2} holds, for any $z_{n+1}\in \Z$, 
\begin{equation}\label{eq-f-Q-2}
f(\cdot,z_{n+1})=f_{n+1}(\cdot).
\end{equation}
This implies that for any $z_{n+1}\in \pi_{n+1}(\supp f)$, $\lambda_{z_{n+1}}$ and  $A_{z_{n+1}}$ are independent of $z_{n+1}$. Thus, 
\begin{equation}
f=\lambda\cdot \chi_{A}
\end{equation}
for some $\lambda>0$ and $A=\prod_{i=1}^{n+1}\pi_i(A)$. So, the conclusion of the theorem is also true for $n+1$. This completes the proof of the theorem.
\end{proof}
Next, we give the proof for Corollary \ref{cor-log-BL}, a logarithmic-type Brascamp-Lieb inequality.
\begin{proof}[Proof of Corollary \ref{cor-log-BL}]
By the Jensen inequality, 
\begin{equation}\label{eq-Jensen}
\begin{split}
\log\|f\|_{\frac{n-1}{n}}=&\frac{n-1}{n}\log\int_{\Z^n}f^{\frac{n}{n-1}}dz_1dz_2\cdots dz_n\\
=&\frac{n-1}{n}\log\int_{\supp f}f^{\frac{n-np+p}{n-1}}f^pdz_1dz_2\cdots dz_n\\
\geq&\frac{n-np+p}{n}\int_{\supp f}f^p\log fdz_1dz_2\cdots dz_n\\
=&\left(\frac{1}{n}+\frac1p-1\right)\int_{\Z^n}f^p\log f^p dz_1dz_2\cdots dz_n.
\end{split}
\end{equation} 
Combining this with \eqref{eq-BL}, we get \eqref{eq-log-BL}. When the equality of \eqref{eq-log-BL} holds, we know that the equality of \eqref{eq-BL} holds. So, by Theorem \ref{thm-BL}, we know that $f=\lambda \chi_A$ with $\lambda>0$ and $A=\prod_{i=1}^n\pi_i(A)$. By that $\|f\|_p=1$, we have $\lambda=|A|^{-\frac1p}$. This completes the proof of the corollary.
\end{proof}

Next, by using Theorem \ref{thm-BL}, we prove Theorem \ref{thm-G-N}.
\begin{proof}[Proof of Theorem \ref{thm-G-N}]
For any $z\in \Z^n$, by that $f\in C_0(\Z^n)$,
\begin{equation*}
\begin{split}
&2f(z)\\
=&\sum_{k=0}^\infty\left(f(z+ke_i)-f(z+(k+1)e_i)\right)+\sum_{k=0}^\infty\left(f(z-ke_i)-f(z-(k+1)e_i)\right)
\end{split}
\end{equation*}
for $i=1,2,\cdots,n.$ So,
\begin{equation}\label{eq-f}
|f(z)|\leq\frac{1}{2}\sum_{k=-\infty}^\infty |f(z+ke_i)-f(z+(k-1)e_i )|=\frac{1}{2}\int_{\Z}|\p_if|dz_i
\end{equation}
which implies that
\begin{equation*}
|f|_i\leq \frac{1}{2}\int_{\Z}|\p_if|dz_i
\end{equation*}
for any $i=1,2,\cdots,n$. Then, by Theorem \ref{thm-BL}, we have
\begin{equation*}
\|f\|_{\frac{n}{n-1}}\leq \prod_{i=1}^n\||f|_i\|_1^\frac1n\leq \frac{1}{2}\prod_{i=1}^n\|\p_i f\|_1^\frac1n.
\end{equation*}
This proves inequality \eqref{eq-G-N}. If the equality holds, we have
\begin{equation}\label{eq-|f|}
|f(z)|=\frac12\int |\p_if|dz_i
\end{equation}
and  by Theorem \ref{thm-BL},
\begin{equation}\label{eq-abs-f}
|f|=\lambda\cdot\chi_{A}
\end{equation}
with $A=\prod_{i=1}^n\pi_i(A)$ and $\lambda>0$.

By \eqref{eq-|f|}, when $f(z)>0$, one has
$$f(z)\geq f(z+e_i)\geq\cdots\geq f(z+ke_i)\geq f(z+(k+1)e_i)\geq\cdots\geq 0$$
and
$$f(z)\geq f(z-e_i)\geq\cdots\geq f(z-ke_i)\geq f(z-(k+1)e_i)\geq\cdots\geq 0$$
for any $i=1,2,\cdots,n$.
So, $f(z)$ is the maximum of $f$ on $\{z+ke_i\ | k\in \Z\}$ and $f(z+ke_i)\geq 0$. By replacing $z$ by $w=z+ke_i$ with $f(w)\neq 0$, we know that $f(w)$ also the the maximum of $f$ on $\{z+ke_i\ | k\in \Z\}$. Thus, when $f(z)>0$, the equality of  \eqref{eq-f} holds if and only if $f(z+ke_i)=f(z)$ for $a_i\leq k\leq b_i$ with some $a_i\leq b_i$ in $\Z$, and $f(z+ke_i)=0$ otherwise. Similarly, we have the same conclusion for $f(z)<0$. Combining this with \eqref{eq-abs-f}, we know that $A$ must be a cuboid in $\Z^n$. This completes the proof of the theorem.
\end{proof}
Finally, we come to prove Corollary \ref{cor-Sobolev}. The proofs for Corollary \ref{cor-log-Sobolev-1} and Corollary \ref{cor-log-Sobolev-2} are straight forward by combining \eqref{eq-Jensen} with Theorem  \ref{thm-G-N} and Corollary \ref{cor-Sobolev} respectively. So, we omit their proofs.  
\begin{proof}[Proof of Corollary \ref{cor-Sobolev}]
By \eqref{eq-G-N} and the AM-GM inequality, we have 
\begin{equation}
\|f\|_{\frac n{n-1}}\leq \frac{1}{2}\prod_{i=1}^n\|\p_if\|_1\leq \frac1{2n}\sum_{i=1}^n\|\p_if\|_1=\frac{1}{2n}\|df\|_1.
\end{equation}
When the equality holds, we know that $f=\lambda\cdot \chi_A$ with $A$ a cuboid in $\Z^n$, and $\|\p_1f\|_1=\|\p_2f\|_1=\cdots=\|\p_nf\|_1$. This implies that $A$ is cube.
\end{proof}

\end{document}